\documentclass[reqno,centertags, 12pt]{amsart}
\usepackage{amsmath,amsthm,amscd,amssymb}
\usepackage{latexsym}
\newcommand{\bbC}{{\mathbb{C}}}

\newcommand{\bbN}{{\mathbb{N}}}

\newcommand{\bbR}{{\mathbb{R}}}

\newcommand{\fre}{{\frak{e}}}

\newcommand{\calR}{{\mathcal R}}

\newcommand{\calT}{{\mathcal T}}

\newcommand{\dott}{\,\cdot\,}

\newcommand{\lb}{\label}
\newcommand{\f}{\frac}
\newcommand{\ul}{\underline}
\newcommand{\ol}{\overline}
\newcommand{\ti}{\tilde  }
\newcommand{\wti}{\widetilde  }

\newcommand{\Lt}{\text{\rm{L}}}
\newcommand{\Md}{\text{\rm{M}}}
\newcommand{\Rt}{\text{\rm{R}}}

\newcommand{\tr}{\text{\rm{Tr}}}

\newcommand{\ess}{\text{\rm{ess}}}
\newcommand{\ac}{\text{\rm{ac}}}

\newcommand{\s}{\text{\rm{s}}}

\newcommand{\supp}{\text{\rm{supp}}}
\newcommand{\intt}{\text{\rm{int}}}

\newcommand{\bi}{\bibitem}

\newcommand{\beq}{\begin{equation}}
\newcommand{\eeq}{\end{equation}}
\newcommand{\ba}{\begin{align}}
\newcommand{\ea}{\end{align}}
\newcommand{\veps}{\varepsilon}

\newcounter{smalllist}
\newenvironment{SL}{\begin{list}{{\rm\roman{smalllist})}}{%
\setlength{\topsep}{0mm}\setlength{\parsep}{0mm}\setlength{\itemsep}{0mm}%
\setlength{\labelwidth}{2em}\setlength{\leftmargin}{2em}\usecounter{smalllist}%
}}{\end{list}}

\DeclareMathOperator{\Real}{Re}
\DeclareMathOperator{\Ima}{Im}

\DeclareMathOperator*{\slim}{s-lim}
\DeclareMathOperator*{\wlim}{w-lim}

\allowdisplaybreaks
\numberwithin{equation}{section}

\newtheorem{theorem}{Theorem}[section]

\newtheorem{proposition}[theorem]{Proposition}
\newtheorem{lemma}[theorem]{Lemma}
\newtheorem{corollary}[theorem]{Corollary}
\theoremstyle{definition}
\newtheorem{example}[theorem]{Example}
\newtheorem{conjecture}[theorem]{Conjecture}

\theoremstyle{remark}
\newtheorem*{remark}{Remark}
\newtheorem*{remarks}{Remarks}
\newtheorem*{definition}{Definition}

\newcommand{\abs}[1]{\lvert#1\rvert}
\newcommand{\jap}[1]{\langle #1 \rangle}
\newcommand{\norm}[1]{\lVert#1\rVert}

\begin{document}
\title{The Nevai Condition}
\author[J.~Breuer, Y.~Last, and B.~Simon]{Jonathan Breuer$^1$, Yoram Last$^{2,4}$,
and Barry Simon$^{3,4}$}

\thanks{$^1$ Mathematics 253-37, California Institute of Technology, Pasadena, CA 91125, USA.
E-mail: jbreuer@caltech.edu}

\thanks{$^2$ Institute of Mathematics, The Hebrew University, 91904 Jerusalem, Israel.
E-mail: ylast@math.huji.ac.il. Supported in part by The Israel Science
Foundation (grant no.\ 1169/06)}

\thanks{$^3$ Mathematics 253-37, California Institute of Technology, Pasadena, CA 91125, USA.
E-mail: bsimon@caltech.edu. Supported in part by NSF grant DMS-0652919}

\thanks{$^4$ Research supported in part
by Grant No.\ 2006483 from the United States-Israel Binational Science
Foundation (BSF), Jerusalem, Israel}

\date{September 4, 2008}

\begin{abstract} We study Nevai's condition that for orthogonal polynomials on the real line,
$K_n(x,x_0)^2 K_n(x_0,x_0)^{-1}\, d\rho (x)\to\delta_{x_0}$ where
$K_n$ is the CD kernel. We prove that it holds for the Nevai class of
a finite gap set uniformly on the spectrum and we provide an example
of a regular measure on $[-2,2]$ where it fails on an interval.
\end{abstract}

\maketitle

\section{Introduction} \lb{s1}

This paper studies material on the borderline of the theory of
orthogonal polynomials on the real line (OPRL) and spectral theory.
Let $d\rho$ be a measure on $\bbR$ of compact but not finite support
and let $P_n(x,d\rho), p_n (x,d\rho)$ be the standard
\cite{SzBk,FrBk,Rice} monic and normalized orthogonal polynomials
for $d\rho$. Let $\{a_n,b_n\}_{n=1}^\infty$ be the Jacobi parameters
defined by
\begin{equation} \lb{1.1}
xp_n(x) = a_{n+1} p_{n+1}(x) + b_{n+1} p_n(x) + a_n p_{n-1}(x)
\end{equation}

The CD (for Christoffel--Darboux) kernel is defined by
\begin{equation} \lb{1.2}
K_n(x,y) =\sum_{j=0}^n p_j(x) p_j(y)
\end{equation}
for $x,y$ real. The CD formula (see, e.g., \cite{CD}) asserts that
\begin{equation} \lb{1.3}
K_n(x,y) =\f{a_{n+1} [p_{n+1}(x) p_n(y) -p_n(x) p_{n+1}(y)]}{x-y}
\end{equation}

The Christoffel variational principle (see \cite{CD}) says that if
\begin{equation} \lb{1.4}
\lambda_n(x_0) =\min \biggl\{\int \abs{Q_n(x)}^2\, d\rho(x)\biggm| \deg Q_n\leq n,\, Q_n(x_0)=1\biggr\}
\end{equation}
then
\begin{equation} \lb{1.5}
\lambda_n(x_0) = \f{1}{K_n(x_0,x_0)}
\end{equation}
and the minimizer is given by
\begin{equation} \lb{1.6}
\wti{Q}_n(x,x_0) = \f{K_n(x,x_0)}{K_n(x_0,x_0)}
\end{equation}
It is quite natural to look at the probability measures
\begin{equation} \lb{1.7}
d\eta_n^{(x_0)}(x) = \f{\abs{\wti{Q}_n(x,x_0)}^2\, d\rho(x)}{\int
\abs{\wti{Q}_n(y,x_0)}^2 \, d\rho(y)}
\end{equation}
so that, by \eqref{1.5} and \eqref{1.6},
\begin{equation} \lb{1.8}
d\eta_n^{(x_0)} = \f{\abs{K_n(x,x_0)}^2\, d\rho(x)}{K_n(x_0,x_0)}
\end{equation}

\begin{definition} We say $d\rho$ obeys a {\it Nevai condition\/} at $x_0$ if and only if
\begin{equation} \lb{1.9}
\wlim_{n\to\infty}\, d\eta_n^{(x_0)} =\delta_{x_0}
\end{equation}
the point mass at $x_0$.
\end{definition}

The name comes from the fact that this condition was studied in the
seminal work of Nevai \cite{Nev79}, who considered the following:

\begin{definition} The {\it Nevai class\/} for $[-2,2]$ is the set of all measures $d\rho$ whose Jacobi
parameters obey
\begin{equation} \lb{1.10}
a_n\to 1 \qquad b_n\to 0
\end{equation}
\end{definition}

$[-2,2]$ is relevant since, by Weyl's theorem on essential spectrum, $[-2,2]$ is the derived set of
$\supp(d\rho)$, that is, the essential spectrum for the Jacobi matrix of $d\rho$.

Nevai proved the following:

\begin{theorem}[Nevai \cite{Nev79}]\lb{T1.1} If $d\rho$ is in the Nevai class for $[-2,2]$, then the
Nevai condition holds for all $x_0$ in $[-2,2]$. Indeed, the limit is uniform for $x_0$ in any  compact
set $K\subset (-2,2)$.
\end{theorem}

The connection of this to spectral theory comes from the relation to the following condition,
sometimes called subexponential growth,
\begin{equation} \lb{1.11}
\lim_{n\to\infty} \f{\abs{p_n(x_0)}^2}{\sum_{j=0}^n \abs{p_j(x_0)}^2} =0
\end{equation}
which we will show (see Proposition~\ref{P2.1}) is equivalent to
\begin{equation} \lb{1.12}
\lim_{n\to\infty}\, \f{(\abs{p_{n-1}(x_0)}^2 + \abs{p_n(x_0)}^2)}{\sum_{j=0}^n \abs{p_j(x_0)}^2}=0
\end{equation}
and to
\begin{equation} \lb{1.13}
\lim_{n\to\infty}\, \f{\abs{p_{n+1}(x_0)}^2}{\sum_{j=0}^n \abs{p_j(x_0)}^2}=0
\end{equation}
We will sometimes need 
\begin{equation} \lb{1.13a}
0 <A_- \equiv \inf_n \, a_n \leq \sup_n \, a_n \equiv A_+ <\infty
\end{equation}
We note $A_+ <\infty$ follows from the assumption that $\supp(d\rho)$ is compact, and if $A_- =0$,
then by a result of Dombrowski \cite{Dom78}, $d\mu$ is purely singular with respect to Lebesgue measure.
Obviously, \eqref{1.13a} holds for the discrete Schr\"odinger case, $a_n\equiv 1$.

The relation of \eqref{1.11} to Nevai's condition is direct:

\begin{theorem}\lb{T1.2} Let \eqref{1.13a} hold. Nevai's condition holds at $x_0$ if and only if
\eqref{1.11} holds.
\end{theorem}

That \eqref{1.11} $\Rightarrow$ \eqref{1.9} is due to Nevai. The converse is new here and appears as
Theorem~\ref{T2.2}.

Equation \eqref{1.11} is, of course, the kind of asymptotic eigenfunction
result of interest to spectral theorists and is susceptible to the
methods of spectral theory. In particular, Theorem \ref{T1.2} shows
that for compactly supported measures, the Nevai condition is
intimately connected with the existence of certain natural sequences
of approximate eigenvectors for the associated Jacobi matrix (for
the relevance of approximate eigenvectors to spectral analysis
see, e.g., \cite{CKL}). Note that \eqref{1.1} and the orthogonality
relation say that, for $\rho$-a.e.\ $x$, the sequence
$(p_0(x),p_1(x),\ldots)$ is a generalized eigenfunction, at $x$, for
the Jacobi matrix, $J$, defined by the parameters
$\{a_n,b_n\}_{n=1}^\infty$ (namely, it is a polynomially bounded
solution of the corresponding eigenvalue equation). Truncations of generalized
eigenfunctions are natural candidates for sequences of approximate
eigenvectors, and Theorem \ref{T1.2} (see \eqref{2.14} in
particular) says that the Nevai condition at $x_0$ is equivalent to
the requirement that truncations of the generalized eigenfunction at $x_0$ 
yield a sequence of approximate
eigenvectors for $J$. In fact, from the point of view of spectral
theory, for compactly supported measures, the Nevai condition is a
simple restatement of this in the ``energy representation'' of the
Jacobi matrix.

This connection is one reason Nevai's condition is interesting---it
has also been used to relate ratios of $\lambda(x_0)$ for the
measures $d\rho(x)$ and $e^{g(x)}\, d\rho(x)$ (Nevai's motivation in
\cite{Nev79}). In this context, it was used by
M\'at\'e--Nevai--Totik \cite{MNT91} to relate CD kernel asymptotics
for OPUC and OPRL and to study OP asymptotics when a Szeg\H{o}
condition fails (see \cite{MNT84,MNT87a,MNT87b}; see also
\cite[Sect.~3.10]{Rice}).

Given the form of the Christoffel variational principle, Nevai's condition seems like something that
must always hold for $x_0\in\supp(d\mu)$. However, \eqref{1.11} also provides a basis for some
counterexamples to Nevai's condition. In Section~\ref{s3} (see Theorem~\ref{T3.2}), we will prove that

\begin{theorem}\lb{T1.3} If
\begin{equation} \lb{1.14}
\liminf_{n\to\infty}\, (\abs{p_n(x_0)}^2 + \abs{p_{n+1}(x_0)}^2)^{1/n} >1
\end{equation}
then \eqref{1.11} fails. \
\end{theorem}

This was the basis for the first counterexample to \eqref{1.11} by Szwarc \cite{Szwarc}; in Section~\ref{s3},
we will see that the Anderson model provides an example of a measure on $[-2,2]$ for which the Nevai condition
fails for Lebesgue a.e.\ $x_0\in [-2,2]$!

Of course, \eqref{1.14} is associated with positive Lyapunov exponent. One might guess that zero Lyapunov
exponent implies \eqref{1.11}. A main impetus for this paper was our realization that this is not true!
Recall that a measure on $[-2,2]$ is called {\it regular\/} if and only if $\lim (a_1\cdots a_n)^{1/n} =1$
(see Stahl--Totik \cite{StT} or Simon \cite{EqMC} for reviews) and that regular measures have zero Lyapunov
exponent, that is, for quasi-every (namely, outside, possibly, a set of zero logarithmic capacity) $x_0\in [-2,2]$,
\begin{equation} \lb{1.14x}
\lim_{n\to\infty}\, \f{1}{n}\, \log(\abs{p_n(x_0)}^2 + \abs{p_{n+1}(x_0)}^2)^{1/2} =0
\end{equation}

In Sections~\ref{s4} and \ref{s5}, we provide two examples of regular measures on $[-2,2]$ for which
Nevai's condition fails for Lebesgue a.e.\ $x_0$ in $\pm (1,2)$. The example in Section~\ref{s4} will
be somewhat simpler but will have no a.c.\ spectrum, while that in Section~\ref{s5} will have pure
a.c.\ spectrum on $(-1,1)$.

We also want to discuss extensions of Nevai's theorem, Theorem~\ref{T1.1}. In this regard, we should
mention some beautiful work of Nevai--Totik--Zhang \cite{NTZ91} and Zhang \cite{Zhang} that already
expanded this. The first paper proved uniform convergence on $[-2,2]$ with an elegant approach; this
was extended in the second paper to $(a_n,b_n)$ approaching a periodic limit. Somewhat earlier,
Lubinsky--Nevai \cite{LN} had proven the Nevai condition for this periodic limit case but only uniformly
on compact subsets of the interior of the spectrum. That paper also has results on subexponential growth
for some cases of measures that do not have compact support, a subject beyond the scope of this paper.
Still later, another proof for the periodic limit case was found by Szwarc \cite{Szwarc2}.

In the past four years, it has become clear that the proper analog of the Nevai class for the periodic
case is not approach to a fixed periodic element but approach to an isospectral torus. We want to prove
that not only can one do this in the periodic case, but on approach to the isospectral torus of, in
general, almost periodic Jacobi matrices that occurs in the general finite gap case. We will also study
the Nevai condition on the a.c.\ spectrum of ergodic Jacobi matrices and for the Fibonacci model.

From what we have said so far, it appears to be a mixed verdict on the Nevai condition since we have
results on when it occurs and also on when it doesn't. But we want to reinterpret the negative results.
These all provide examples where it is not true that the Nevai condition holds {\it everywhere\/} on the
topological support of $d\rho$ or even that it fails for Lebesgue a.e.\ $x_0$ on the support. We believe
it is likely that the following more refined property could be true:

\begin{conjecture}\lb{Con1.4} For any measure, $d\rho$, with compact support, the Nevai condition holds
for $d\rho$-a.e.\ $x_0$ in $\supp(d\rho)$.
\end{conjecture}

Here is a summary of the contents of the rest of the paper. In
Section~\ref{s2}, we prove Theorem~\ref{T1.2}. In Section~\ref{s3},
we discuss cases with positive Lyapunov exponent, including a proof
of Theorem~\ref{T1.3}. Sections~\ref{s4} and \ref{s5} present the
details of the examples of regular measures where the Nevai
condition fails for Lebesgue a.e.\ $x_0$ in a particular open
subset. Section~\ref{s6} presents our version of the NTZ approach,
as preparation for our discussion of the finite gap Nevai class in
Section~\ref{s7}. In Section~\ref{snew8}, we will relate
\eqref{1.11} to the absence of $\ell^2$ solutions of classical right
limits and so recover the results of Section~\ref{s7} and even more
(including Fibonacci models). The ideas of Section~\ref{snew8} seem
to be more generally applicable than Section~\ref{s6}, but the
constants in Section~\ref{s6} are more explicit. Section~\ref{s8}
has some final remarks, including a discussion of
Conjecture~\ref{Con1.4} and a second conjecture (Conjecture~\ref{Con9.5}).
Section~\ref{s8} also notes that if $\f{1}{n} K_n(x_0,x_0)$ has a nonzero limit,
then the Nevai condition holds, and so links this to recent work on that question.
Thus, for those interested in ergodic Schr\"odinger operators, by
the end of this paper, we will have proven the Nevai condition in
the Fibonacci model uniformly on the spectrum and for ergodic models
with a.c.\ spectrum, Lebesgue a.e.\ on the essential support of the a.c.\ spectrum.

\medskip
We would like to thank David Damanik and Svetlana Jitomirskaya for
valuable correspondence and Benjamin Weiss for valuable discussions.
J.B.\ and B.S.\ would like to thank
Ehud~de~Shalit for the hospitality of Hebrew University where some
of this work was done.

\section{Subexponential Decay} \lb{s2}

We begin with the equivalence of \eqref{1.11}, \eqref{1.12}, \eqref{1.13}, and more:

\begin{proposition}\lb{P2.1} Let $c_0,c_1,c_2, \dots$ be a sequence of nonnegative numbers and
\begin{equation} \lb{2.1}
S_n =\sum_{j=0}^n c_j
\end{equation}
Then the following are equivalent:
\begin{alignat}{2}
&\text{\rm{(i)}} \qquad && c_n/S_n\to 0 \lb{2.2} \\
&\text{\rm{(ii)}} \qquad && S_n/S_{n+1} \to 1 \lb{2.3} \\
&\text{\rm{(iii)}} \qquad && c_{n+1}/S_n \to 0 \lb{2.4} \\
&\text{\rm{(iv)}} \qquad && (c_n + c_{n+1})/S_n \to 0 \lb{2.6} \\
&\text{\rm{(v)}} \qquad && (c_{n-1} + c_n)/S_n \to 0 \lb{2.7}
\end{alignat}
\end{proposition}

\begin{remarks} 1. The relevance, of course, is to
\begin{equation} \lb{2.8}
c_n = p_n(x_0)^2
\end{equation}

\smallskip
2. If $c_n=e^{n^2}$, $c_n/S_n\to 1$ and $c_{n-1}/S_n\to 0$, so (i) is not equivalent to $c_{n-1}/S_n\to 0$.
\end{remarks}

\begin{proof} \ul{(i) $\Leftrightarrow$ (ii)}. \ We have
\begin{equation} \lb{2.9}
S_{n-1}/S_n = 1 - c_n/S_n
\end{equation}
so
\begin{equation} \lb{2.10}
\text{(i)} \Leftrightarrow S_{n-1}/S_n \to 1
\end{equation}
which, by renumbering indices, is equivalent to (ii).

\smallskip
\ul{(ii) $\Leftrightarrow$ (iii)}. \ We have
\begin{equation} \lb{2.11}
S_{n+1}/S_n = 1 + c_{n+1}/S_n
\end{equation}
from which (ii) is equivalent to (iii).

\smallskip
\ul{(iv) or (v) $\Rightarrow$ (i)}. \ Immediate, since for $j=\pm 1$,
\begin{equation} \lb{2.12}
0 \leq c_n/S_n \leq (c_n + c_{n+j})/S_n
\end{equation}

\smallskip
\ul{(i) $\Rightarrow$ (iv)}. \ Immediate from (i) $\Rightarrow$ (iii).

\smallskip
\ul{(i) $\Rightarrow$ (v)}. \ Since (i) $\Rightarrow S_{n-1}/S_n\to 1$, (i) $\Rightarrow c_{n-1}/S_n
\to 0$, from which (v) is immediate.
\end{proof}

\begin{theorem}[$\Rightarrow$ Theorem~\ref{T1.2}] \lb{T2.2} Suppose that \eqref{1.13a} holds. Nevai's
condition is equivalent to
\begin{equation} \lb{2.13}
\int (x-x_0)^2\, d\eta_n^{(x_0)} (x) \to 0
\end{equation}
and \eqref{2.13} holds if and only if \eqref{1.11} holds.
\end{theorem}

\begin{proof} By \eqref{1.8}, \eqref{1.3}, and the orthogonality of $p_n$ to $p_{n+1}$,
\begin{equation} \lb{2.14}
\int (x-x_0)^2\, d\eta_n^{(x_0)}(x) = \f{a_{n+1}^2 [p_n(x_0)^2 + p_{n+1}(x_0)^2]}{K_n(x_0,x_0)}
\end{equation}
By \eqref{1.13a} and (i) $\Leftrightarrow$ (iv) in Proposition~\ref{P2.1},
\[
\text{\eqref{2.13}} \Leftrightarrow \f{p_n(x_0)^2 + p_{n+1}(x_0)^2}{K_n(x_0,x_0)} \to 0
\Leftrightarrow\text{\eqref{1.11}}
\]

For measures, $\{ \nu_n \}$, all supported in a fixed compact,
$\nu_n\overset{w}{\longrightarrow} \delta_{x_0} \Leftrightarrow \int
(x-x_0)^2\, d\nu_n \to 0$.
\end{proof}

In the example in Remark~2 after Proposition~\ref{P2.1},
$(c_{n-2}+c_{n-1})/S_n \to 0$ but \eqref{2.2} fails. However, this
cannot happen for the case $c_n=p_n (x_0)$: Note that if
\eqref{1.13a}
 holds,
\begin{equation}\lb{2.14x}
\abs{p_n(x_0)} \leq A_-^{-1} [A_+ + \abs{x_0} + \sup_n \, \abs{b_n}] \,
[\abs{p_{n-2}(x_0)} + \abs{p_{n-1}(x_0)}]
\end{equation}
since
\begin{equation}\lb{2.15}
p_n(x_0) = a_n^{-1} ((x_0-b_n) p_{n-1}(x_0) -a_{n-1} p_{n-2}(x_0))
\end{equation}
and note the following obvious fact:

\begin{proposition}\lb{P2.3} Under the hypotheses and notation of Proposition~\ref{P2.1}, if there is a
constant $K$ so that
\begin{equation}\lb{2.16}
c_n \leq K (c_{n-2} +c_{n-1})
\end{equation}
then \eqref{2.2} is equivalent to
\begin{equation}\lb{2.17}
\f{c_{n-2}+c_{n-1}}{S_n} \to 0
\end{equation}
\end{proposition}

\section{Positive Lyapunov Exponent} \lb{s3}

We begin by proving a contrapositive of Theorem~\ref{T1.3}:

\begin{proposition}\lb{P3.1} If \eqref{1.11} holds, then
\begin{equation} \lb{3.1}
\limsup_{n\to\infty}\, K_n(x_0,x_0)^{1/n} \leq 1
\end{equation}
So, in particular,
\begin{equation} \lb{3.2}
\limsup_{n\to\infty}\, (\abs{p_n(x_0)}^2 + \abs{p_{n+1}(x_0)}^2)^{1/n} \leq 1
\end{equation}
\end{proposition}

\begin{proof} Given $\veps$, pick $N$ so for $n\geq N$,
\begin{equation} \lb{3.3}
\f{\abs{p_n(x_0)}^2}{K_n(x_0,x_0)} \leq \veps
\end{equation}
Then
\begin{equation} \lb{3.4}
\f{K_{n-1}(x_0,x_0)}{K_n(x_0,x_0)}\geq 1-\veps
\end{equation}
so, for $n\geq N$,
\begin{equation} \lb{3.5}
K_n(x_0,x_0) \leq (1-\veps)^{-1} K_{n-1} (x_0,x_0)
\end{equation}
which implies
\begin{equation} \lb{3.6}
\limsup\, K_n(x_0,x_0)^{1/n} \leq (1-\veps)^{-1}
\end{equation}

Since $\veps$ is arbitrary, \eqref{3.1} holds. Thus, since
\begin{equation} \lb{3.7}
(\abs{p_n(x_0)}^2 + \abs{p_{n+1}(x_0)}^2)^{1/n} \leq K_{n+1}(x_0,x_0)^{1/n}
\end{equation}
we get \eqref{3.2}.
\end{proof}

\begin{theorem}[$\equiv$ Theorem~\ref{T1.3}] \lb{T3.2} \eqref{1.14} $\Rightarrow$ not \eqref{1.11}.
\end{theorem}

\begin{proof} As noted, this is a contrapositive of the last statement in Proposition~\ref{P3.1}.
\end{proof}

Recall that the transfer matrix at $x_0$ is defined by
\begin{align}
T_n(x_0) &= A_n(x_0) \dots A_1(x_0) \lb{3.8} \\
A_j(x_0) &= \f{1}{a_j}
\begin{pmatrix}
x_0-b_j & -1 \\
a_j^2 & 0
\end{pmatrix}  \lb{3.9}
\end{align}
so $\det(A_j) =\det(T_n) =1$ and
\begin{equation} \lb{3.10}
\binom{p_n(x_0)}{a_n p_{n-1}(x_0)} = T_n(x_0) \binom{1}{0}
\end{equation}

Recall also that one says the Lyapunov exponent exists if
\begin{equation} \lb{3.11}
\gamma(x_0) = \lim_{n\to\infty}\, \f{1}{n}\, \log \norm{T_n(x_0)}
\end{equation}
exists. The Ruelle--Osceledec theorem (see, e.g., \cite[Thm.~10.5.29]{OPUC2}) says that if $\gamma(x_0)>0$,
then there is a one-dimensional subspace, $V$\!, of $\bbC^2$, so $u\in V\setminus\{0\}$ implies
\begin{equation} \lb{3.12}
\lim_{n\to\infty}\, \norm{T_n(x_0)u}^{1/n} =e^{-\gamma(x_0)}
\end{equation}
and if $u\in\bbC^2\setminus V$\!, then
\begin{equation} \lb{3.13}
\lim_{n\to\infty}\, \norm{T_n(x_0) u}^{1/n} = e^{\gamma(x_0)}
\end{equation}

Thus,

\begin{corollary} \lb{C3.3} Let \eqref{1.13a} hold. If
$\gamma(x_0) >0$, then either
$x_0$ is a pure point of $d\rho$ or else \eqref{1.11} fails.
\end{corollary}

\begin{proof} If \eqref{3.12} holds for $u=\binom{1}{0}$, then by \eqref{3.10}, $\abs{p_n(x_0)}\leq
Ce^{-\gamma(x_0)n/2}$ so $p_n\in\ell^2$ and $x_0$ is a pure point. So if $x_0$ is not a pure point, then
\eqref{3.13} holds, so since $A_- >0$, \eqref{3.13} implies
\begin{equation} \lb{3.14}
\lim_{n\to\infty}\, (\abs{p_n(x_0)}^2 + \abs{p_{n+1}(x_0)}^2)^{1/n} = e^{\gamma(x_0)} >1
\end{equation}
and so \eqref{1.11} fails.
\end{proof}

\begin{example}[Szwarc \cite{Szwarc}] \lb{E3.4} This example is mainly of historical interest. This is a
modification of \cite{Szwarc}, but uses the key notion of having an isolated point of $\sigma_\ess (J)$,
which is not an eigenvalue. We begin by noting that the whole-line Jacobi matrix, $J_\infty$, with
$a_n\equiv 1$, $b_j=0$ ($j\neq 0$) and $b_0=\f32$ has $E=\f52$ ($=2+\f12$) as an eigenvalue with
eigenfunction $(\f12)^{\abs{n}}$, so $\sigma(J_\infty)=[-2,2]\cup\{\f52\}$.

Now let $J$ be the one-sided Jacobi matrix with $a_n\equiv 1$ and
\begin{equation} \lb{3.15}
b_n = \begin{cases}
\beta & n=1 \\
\f32 & n=k^2, \, k=2,3,\dots \\
0 & n\neq k^2 \text{ all } k
\end{cases}
\end{equation}
where $\beta$ will be adjusted below.

Standard right-limit theorems (see, e.g., \cite{LaSi2}) imply
\[
\sigma_\ess (J) = [-2,2]\cup\{\tfrac52\}
\]
so, in particular, $\f52\in\sigma(J)$. Moreover, since the nonzero $b_n$'s are of zero density, it is easy to
see that
\begin{align}
\gamma(\tfrac52) &= \log\biggl(\text{spectral radius of }
\begin{pmatrix} \f52 & -1 \\
1 & 0
\end{pmatrix} \biggr) \notag \\
&= \log(2) \lb{3.16}
\end{align}
since $\left( \begin{smallmatrix} \f52 & -1 \\ 1 & 0 \end{smallmatrix}\right)$ has eigenvalues $2$ and $\f12$.

It is easy to see that there is exactly one choice of $\beta$ for which $\f52$ is an eigenvalue of $J$. For
any other choice, Corollary~\ref{C3.3} is applicable, and so the Nevai condition fails at $\f52\in\sigma(J)$.
\qed
\end{example}

\begin{example}\lb{E3.5} Let $a_n\equiv\f12$ and $b_n(\omega)$ be i.i.d.\ random variables uniformly distributed
in $[-1,1]$. This is an Anderson model for which it is well-known that for a.e.\ choice of $\omega$ (see, e.g.,
\cite{CFKS}), the associated measure $d\rho_\omega$ has pure point spectrum filling $[-2,2]$. Moreover, for
a.e.\ $\omega$ and quasi-every $x_0\in [-2,2]$, one has Lyapunov exponent $\gamma(x_0) >0$. Since the set of
eigenvalues is countable, for quasi-every $x_0\in [-2,2]$, Corollary~\ref{C3.3} is applicable. Thus,
$\sigma(J_\omega)=[-2,2]$, but for quasi-every $x_0\in [-2,2]$, the Nevai condition fails.
\qed
\end{example}

\section[A Regular Measure]{A Regular Measure for Which the Nevai Condition Fails
on a Set of Positive Lebesgue Measure} \lb{s4}

\begin{example}\lb{E4.1} Let $J$ be a Jacobi matrix with $b_n\equiv 0$ and $a_n$ described as follows. Partition
$\{1,2,\dots\}$ into successive blocks $A_1, C_1, A_2, C_2 \dots$, where
\begin{equation} \lb{4.1}
\#(A_j) =3^{j^2} \qquad \#(C_j) =2^{j^2}
\end{equation}
On $A_j$, $a_n\equiv 1$ and on $C_j$, $a_n\equiv \f12$. Since the $3^{j^2}$ blocks are much larger than the
$2^{j^2}$ blocks,
\begin{equation} \lb{4.2}
\lim_{n\to\infty}\, (a_1 \dots a_n)^{1/n}=1
\end{equation}
and since $\abs{a_n}\leq 1$, $\norm{J}\leq 2$ so $\sigma(J)\subset
[-2,2]$. By a simple trial function argument, $[-2,2]\subset
\sigma(J)$. Thus, $J$ is regular for $[-2,2]$, but, of course, not
in Nevai class since $a_n\nrightarrow 1$.
\end{example}

We will prove that

\begin{theorem}\lb{T4.2} For Lebesgue a.e.\ $x_0\in [-2,2]\setminus [-1,1]$, the Nevai condition fails.
\end{theorem}

The intuition, which we will implement, goes as follows: For $x_0$ in
the specified set, in $A_j$ regions, solutions of the eigenfunction
equation are linear combinations of plane waves. So in such regions,
$p_n(x_0)$ hardly grows or decays. In $C_j$ regions, they are linear
combinations of growing and decaying exponentials, so usually, the
growing exponentials will win and the $p_n(x_0)$ will grow
exponentially. $C_{j+1}$ is much bigger than $C_j$ (indeed,
$\#(C_{j+1})=2\,4^j \#(C_j)$), so at the center of $C_{j+1}$,
$\abs{p_n(x_0)}^2$ will be much bigger than
$\sum_{k\in\cup_{\ell=1}^j (A_\ell\cup C_\ell)\cup A_{j+1}}
\abs{p_k(x_0)}^2$ and comparable to $\sum_{\substack{k\in C_{j+1} \\
k\leq n-1}} \abs{p_k(x_0)}^2$, which will prevent \eqref{1.11} from
holding.

We will say more about the general strategy  shortly, but we first implement the initial step:

\begin{proposition}\lb{P4.3} Suppose for a given $x_0$, there are $C,D$ and $\alpha,\beta >0$ so that
\begin{SL}
\item[{\rm{(i)}}] For all $j$ and $n\in C_j \cup A_{j+1}$,
\begin{equation} \lb{4.3}
\abs{p_n(x_0)} \leq C^{j+1} \exp (\alpha 2^{j^2})
\end{equation}
\item[{\rm{(ii)}}] For all large $j$ and $n_j$, the ``center'' of $C_j$,
we have
\begin{equation} \lb{4.4}
(\abs{p_{n_j-1}(x_0)}^2 + \abs{p_{n_j}(x_0)}^2)^{1/2} \geq D^{-(j+1)} \exp (\beta 2^{j^2})
\end{equation}
\end{SL}
Then \eqref{1.11} {\rm{(}}and so the Nevai condition{\rm{)}} fails at $x_0$.
\end{proposition}

\begin{remarks} 1. In fact, the proof shows that \eqref{4.4} must only hold for infinitely many $j$'s.

\smallskip
2. Since $\#(C_j)$ is even, it does not have a strict center. By ``center,'' we mean one half unit prior to
the midpoint.
\end{remarks}

\begin{proof} Suppose \eqref{1.11} holds. Then, since \eqref{1.13} holds, for any $\veps>0$, there exists
$N(\veps)$, so for $n\geq N(\veps)$,
\begin{equation} \lb{4.5}
\abs{p_{n+1}(x_0)}^2 \leq \veps \sum_{k=1}^n\, \abs{p_k(x_0)}^2
\end{equation}
so
\begin{equation} \lb{4.6}
\sum_{k=1}^{n+1}\, \abs{p_k(x_0)}^2 \leq (1+\veps) \sum_{k=1}^n \, \abs{p_k(x_0)}^2
\end{equation}
and thus, for $n\geq m\geq N(\veps)$,
\begin{equation} \lb{4.7}
\abs{p_{n-1}(x_0)}^2 + \abs{p_n(x_0)}^2 \leq (1+\veps)^{n-m} \sum_{k=1}^m \, \abs{p_k(x_0)}^2
\end{equation}

Now, suppose $C_j$ is such that its leftmost point, $m_j+1$, has
$m_j\geq N(\veps)$ and let $n_j$ be the center of $C_j$. Then by
\eqref{4.3},
\begin{equation} \lb{4.8}
\sum_{k=1}^{m_j}\, \abs{p_k(x_0)}^2 \leq C^{2j} \exp(2\alpha 2^{(j-1)^2}) \biggl[\, \sum_{k=1}^j
2^{k^2} + 3^{k^2}\biggr]
\end{equation}
and by \eqref{4.4},
\begin{equation} \lb{4.9}
\abs{p_{n_j-1}(x_0)}^2 + \abs{p_{n_j}(x_0)}^2 \geq D^{-2(j+1)} \exp (2\beta 2^{j^2})
\end{equation}

Pick $\veps$ so $(1+\veps)\leq e^{2\beta}$. Since $n_j-m_j=\f12 (2^{j^2})$, \eqref{4.7} says that so long as
$m\geq N(\veps)$,
\begin{equation} \lb{4.10}
D^{-2(j+1)} \exp (2\beta 2^{j^2}) \leq \exp (\beta 2^{j^2}) C^{2j}
\exp (2\alpha 2^{(j-1)^2}) (2j 3^{j^2})
\end{equation}
Since $2\beta >\beta + 4\,4^{-j}\alpha$ for $j$ large, \eqref{4.10} cannot hold for large $j$. This
contradiction implies that \eqref{4.5} cannot hold for this value of $\veps$, and thus, \eqref{1.11} fails.
\end{proof}

The upper bound, \eqref{4.3}, will be easy from transfer matrix arguments. The lower bound, \eqref{4.4},
is much more subtle. Indeed, it implies that for those $n$'s, asymptotically $\abs{p_n(x_0)}\geq n^\ell$
for all $\ell$. On the other hand, for a.e.\ $x_0$ with respect to a spectral measure, $p_n(x_0)$ is
polynomially bounded. Thus, \eqref{4.3} must fail for a dense set of $x_0$'s! Fortunately, this kind
of problem has been faced before in spectral theory contexts and we will be able to borrow a technique
from Jitomirskaya--Last \cite{JL1}. We turn first to the upper bound.

For $x_0\in (-2,2)\setminus [-1,1]$, define $\theta(x_0), \eta(x_0)$ by
\begin{equation} \lb{4.11}
2\cos\theta(x_0) =x_0 \qquad \cosh \eta(x_0)=\abs{x_0}
\end{equation}
Let $Q(x_0), R(x_0)$ be the one-unit transfer matrices for $b_n\equiv 0$, $a_n\equiv 1$ and $b_n\equiv 0$,
$a_n\equiv \f12$. Then $Q(x_0)$ has $e^{\pm i\theta_0}$ as eigenvalues and $R(x_0)$ has eigenvalues
$e^{\pm\eta(x_0)}$ if $x_0 >0$ and $-e^{\pm\eta(x_0)}$ if $x_0 <0$. It follows that for some constants
$c(x_0)$, $d(x_0)$,
\begin{align}
\norm{Q(x_0)^k} &\leq c(x_0) \lb{4.12} \\
\norm{R(x_0)^k} &\leq d(x_0) e^{k\eta(x_0)} \lb{4.13}
\end{align}
with $c,d$ bounded uniformly on compacts of $(-2,2)\setminus [-1,1]$. This lets us prove that

\begin{proposition}\lb{P4.4} For any compact subset $K$ of $(-2,2)\setminus [-1,1]$, there are constants
$C,\alpha$ so that for all $n,j$ and $x_0\in K$ with $n\in C_j \cup
A_{j+1}$, we have
\begin{equation} \lb{4.14}
\norm{T_n(x_0)} \leq C^{j+1} \exp (\alpha 2^{j^2})
\end{equation}
\end{proposition}

\begin{remark} Since $p_n(x_0)$ is a matrix element of $T_n$, we immediately have \eqref{4.3}.
\end{remark}

\begin{proof} For $n\in A_{j+1}$, $T_n(x_0)$ is a product of $(j+1)$
factors of products of $Q(x_0)$ and
$j$ factors of products of $R(x_0)$. Thus, by \eqref{4.12} and \eqref{4.13},
\begin{equation} \lb{4.15}
\norm{T(x_0)} \leq c(x_0)^{j+1} d(x_0)^j \exp \biggl(\eta(x_0) \sum_{\ell=1}^j 2^{\ell^2}\biggr)
\end{equation}
For $n\in C_j$, the estimate is similar, but $c(x_0)^{j+1}$ is replaced by $c(x_0)^j$ and $\sum_{\ell=1}^j
2^{\ell^2}$ by a smaller sum.

Since $c$, $d$, and $\eta$ are bounded on $K$ and
\begin{equation} \lb{4.16}
\sum_{\ell=1}^j 2^{\ell^2} \leq 2^{\ell^2} [1+\tfrac12 + \tfrac14 + \dots ] = 2\, 2^{\ell^2}
\end{equation}
we obtain \eqref{4.14}.
\end{proof}

To get the lower bound following Jitomirskaya--Last \cite{JL1}, we need to consider Weyl solutions and
Green's functions. For $\Ima z>0$, there is a unique solution $u_n^+(z)$ of
\begin{equation} \lb{4.17}
a_n u_{n+1} + (b_n-z) u_n + a_{n-1} u_{n-1} =0
\end{equation}
defined for $n=1,2,\dots$ (with $a_0\equiv 1$) which is $\ell^2$ at infinity and normalized by
\begin{equation} \lb{4.18}
u_0^+(z)=-1
\end{equation}
This is the Weyl solution. The spectral theorist's Green's function (different from the Green's function
of potential theory!) is defined for $n,m\geq 1$ by
\begin{equation} \lb{4.19}
G_{nm}(z) =\jap{\delta_n, (J-z)^{-1}\delta_m}
\end{equation}
Then $G_{nm}=G_{mn}$ and, for $1\leq n\leq m$,
\begin{equation} \lb{4.20}
G_{nm}(z)=p_{n-1}(z) u_m^+ (z)
\end{equation}

As usual, the Wronskian is constant and, by $a_0\equiv 1$ and \eqref{4.18} plus $p_{-1}=0$, $p_0=1$, this
constant is $1$. So for all $n=0,1,2,\dots$,
\begin{equation} \lb{4.21}
a_n (u_{n+1}^+(z) p_{n-1}(z) - u_n^+(z) p_n(z)) =1
\end{equation}

By \eqref{4.19}, $G_{nm}$ is the Borel transform of a signed measure, and so for Lebesgue a.e.\ $x_0$,
it has boundary values $G_{nm}(x_0 +i0)$. In particular, since $p_0=1$, $u_n^+ =G_{1n}$ has a.e.\
boundary values, $u_n^+ (x_0 +i0)$. \eqref{4.20} and \eqref{4.21} still hold for $z=x_0 +i0$.

In particular, for our example where $a_n$ is $1$ or $\f12$, \eqref{4.21} and the Schwarz inequality
imply that
\begin{equation} \lb{4.22}
(p_n(x_0)^2 + p_{n-1}(x_0)^2) \geq (\abs{u_{n+1}^+ (x_0 + i0)}^2 + \abs{u_n^+ (x_0+i0)}^2)^{-1}
\end{equation}
So to get an exponentially growing lower bound on $p_n(x_0)^2 + p_{n-1}(x_0)^2$, we only need to get an
exponentially decaying upper bound on $\abs{u_n^+ (x_0+i0)}$ and $\abs{u_{n+1}^+ (x_0 + i0)}$.

Now fix $1<k<\ell<\infty$ and define $\ti J$ to be the Jacobi matrix
obtained by replacing $a_k$ and $a_\ell$ by $0$. Thus, under
\begin{align}
\ell^2 (\{1,2,\dots\}) &=\ell^2 (\{1, \dots, k\}) \oplus \ell^2 (\{k+1, \dots, \ell\}) \notag \\
& \qquad \qquad \oplus \ell^2 (\{\ell+1, \ell+2, \dots,\}) \lb{4.23} \\
\ti J &=J_\Lt \oplus J_\Md \oplus J_\Rt \lb{4.24}
\end{align}
($\Lt,\Md,\Rt$ for left, middle, right).

\begin{proposition}\lb{P4.5} Let $\wti G$ be the Green's function for $\ti J$. Let
\begin{equation} \lb{4.25}
k+1\leq n\leq \ell
\end{equation}
Then
\begin{equation} \lb{4.26}
G_{1n} = -a_k G_{1k} \wti G_{k+1\, n} - a_\ell G_{1\, \ell+1} \wti
G_{\ell n}
\end{equation}
\end{proposition}

\begin{proof} Define
\begin{equation} \lb{4.27}
\Gamma = J-\ti J
\end{equation}
Then
\begin{equation} \lb{4.28}
(J-z)^{-1} = (\ti J-z)^{-1} - (J-z)^{-1} \Gamma (\ti J-z)^{-1}
\end{equation}
so
\begin{equation} \lb{4.29}
G_{1n} = \wti G_{1n} -\sum_{m,r} G_{1m} \Gamma_{mr} \wti G_{rn}
\end{equation}

Since $1$ and $n$ lie in different blocks in the direct sum
\eqref{4.24}, $\wti G_{1n}=0$. $\Gamma$ is a rank four operator;
since $\wti G_{kn}=\wti G_{\ell+1\, n}=0$, two terms in the sum in
\eqref{4.29} vanish. The result is \eqref{4.26}.
\end{proof}

\begin{proposition} \lb{P4.6} Let $J^{(k)}$ be the $k\times k$ matrix with $0$'s on the diagonal and $\f12$
in each of the two principal off-diagonals. Let $G^{(k)}(z)$ be the matrix $(J^{(k)}-z)^{-1}$. Then for
$x_0\notin [-1,1]$ and $m\leq n$,
\begin{equation} \lb{4.30}
G_{mn}^{(k)}(x_0) = \f{2(w^{-m} -w^m)(w^{-(k+1-n)} -w^{(k+1-n)})}{(w^{-1}-w)(w^{-(k+1)}-w^{(k+1)})}
\end{equation}
where $\abs{w}>1$ and solves
\begin{equation} \lb{4.31}
w^{-1} +w =2x_0
\end{equation}
In particular, for any compact $K\in (-2,2)\setminus [-1,1]$, there are $\gamma >0$ and $C$ so that for
all $x_0\in K$ and all $k=2r$,
\begin{equation} \lb{4.32}
\abs{G_{1n}^{(2r)}(x_0)} \leq Ce^{-\gamma r}
\end{equation}
for
\begin{equation} \lb{4.33}
n=r, r-1
\end{equation}
\end{proposition}

\begin{proof} $w^n$ and $w^{-n}$ solve
\begin{equation} \lb{4.36}
\tfrac12\, (u_{n+1}+u_{n-1}) = x_0u_n
\end{equation}
so $w^{-n}-w^n$ solves \eqref{4.36} with $u_0=0$ boundary condition,
while $w^{-(k+1-n)} -w^{(k+1-n)}$ solves it with $u_{k+1}=0$
boundary condition. The numerator in \eqref{4.30} is twice their
product and the denominator twice their Wronskian, proving
\eqref{4.30}. \eqref{4.32} follows by noting that the dominant term
in the numerator is $w^r$, while in the denominator, $w^{2r}$.
\end{proof}

Thus, in \eqref{4.26}, where $k+1$ and $\ell$ are taken to be the
edge of a $C_j$ block and $n$ to be the center or one less, the
$\wti G$ terms are exponentially small. We thus need estimates on
the set of $x_0$ for which $G_{1k}(x_0)$ can be large. Here we
recall Loomis's theorem:

\begin{theorem}[Loomis \cite{Loo}]\lb{T4.6} Let $\mu$ be a complex measure on $\bbR$ of
finite total variation $\norm{\mu}$. Let
\begin{equation} \lb{4.36a}
F_\mu(x) =\lim_{\veps\downarrow 0} \int \f{d\mu(y)}{y-(x+i\veps)}
\end{equation}
which exists for Lebesgue a.e.\ $x$. Then with $\abs{\dott}=$ Lebesgue measure,
\begin{equation} \lb{4.37}
\abs{\{x\mid \abs{F_\mu(x)} >M\}} \leq \f{C\norm{\mu}}{M}
\end{equation}
for a universal constant $C$.
\end{theorem}

\begin{remarks} 1. The history is complicated and is partly described in Loomis \cite{Loo}. The result for
measures $\mu$ which are absolutely continuous is due earlier to Kolmogorov.

\smallskip
2. For $d\mu= f(x)\, dx$ with $f\in L^1$, the optimal constant was found by Davis \cite{Dav}. For purely
singular positive measures, the result is essentially due to Boole \cite{Boole} with constant $C=2$, which is
optimal---indeed, one has equality.

\smallskip
3. In \cite{JL1}, they only used Boole's equality since their
measure is purely singular. That is true here also, but not in the
next section. In any event, it is useful to know pure singularity is
not needed a priori, although it follows on $[-2,2]\setminus [-1,1]$
from the estimates here.
\end{remarks}

\begin{proposition} \lb{P4.7} Let $n_1 < n_2 < \dots$ be an arbitrary sequence of indices. Let $\delta >0$.
Then for a.e.\ $x_0\in\bbR$, $\exists J(x_0)$ so that
\begin{equation} \lb{4.38}
j > J(x_0) \Rightarrow \abs{G_{1n_j}(x_0+i0)} \leq e^{\delta j}
\end{equation}
\end{proposition}

\begin{proof} This is a standard Borel--Cantelli argument. Let $\chi_j$ be the characteristic function of
$\{x_0\mid\abs{G_{1n_j}(x_0)} > e^{\delta j}\}$. Since $G_{1n_j}(x_0+i0)$ is of the form \eqref{4.36a}
for a measure of variation at most $1$,
\begin{equation} \lb{4.39}
\int \chi_j (x)\, dx \leq Ce^{-\delta j}
\end{equation}
Thus,
\begin{equation} \lb{4.40}
\int \sum_{j=1}^\infty \chi_j(x)\, dx <\infty
\end{equation}
which implies that for a.e.\ $x_0$, $\sum_{j=1}^\infty \chi_j(x_0) <\infty$. Since each $\chi_j(x_0)$ is
$0$ or $1$, only finitely many are nonzero, that is, for all large $j$,
\eqref{4.38} holds.
\end{proof}

\begin{proof}[Proof of Theorem~\ref{T4.2}] By Propositions~\ref{P4.3} and \ref{P4.4},
we only need to prove that \eqref{4.4} holds for
Lebesgue a.e.\ $x_0$. By \eqref{4.22}, it suffices to prove
exponentially decaying upper bounds on $u_{n_j}^+, u_{n_j+1}^+$ (for
the same $n_j$ as \eqref{4.4}). As noted, $u_n^+ =G_{1n}$, so it
suffices to prove exponentially decaying upper bounds on $G_{1\,
n_j}, G_{1\, n_j+1}$.

We use \eqref{4.26} with $k,\ell$ the lower and upper edge of the $C_j$ block. By Proposition~\ref{P4.6},
\begin{equation} \lb{4.38x}
\abs{\wti G_{k+1\, n_j}} + \abs{\wti G_{k+1\, n_j+1}} + \abs{\wti
G_{\ell\, n_j}} + \abs{\wti G_{\ell\, n_j+1}} \leq C\exp (-\gamma
2^{j^2})
\end{equation}
for some $\gamma >0$.

Thus, the result follows from Proposition~\ref{P4.7}, for eventually
each of the complementary Green's functions in \eqref{4.26} is
bounded by $\exp(\f12\gamma j)\leq\exp(\f12\gamma 2^{j^2})$ and
\eqref{4.4} holds with $\beta = \f12 \gamma$.
\end{proof}

Finally, we want to note that $J$ has a two-sided right limit which
has $a_n=\f12$ for $n\leq -1$ and $a_n=1$ for $n\geq 0$. There is no
set of positive Lebesgue measure on which $J$ is reflectionless, so,
by a theorem of Remling \cite{Rem}, $J$ has purely singular
spectrum.

\section{A Regular Measure With Some A.C.\ Spectrum} \lb{s5}

As we noted at the end of the last section, Example~\ref{E4.1} has no a.c.\ spectrum. Of course, if
Conjecture~\ref{Con1.4} is true, then Lebesgue a.e.\ on the a.c.\ spectrum, the Nevai condition holds.
So an example like Example~\ref{E4.1} cannot have a.c.\ spectrum on $[-2,2]\setminus [-1,1]$ but it
can on $[-1,1]$. The example in this section shows that a.c.\ spectrum is indeed possible on $[-1,1]$.

\begin{example}\lb{E5.1} Let $J$ be a Jacobi matrix with $b_n\equiv 0$ and $a_n$ described as follows:
Partition $\{1,2,\dots\}$ into successive blocks $A_1,B_1,C_1,D_1,A_2,B_2, \dots$, where
\begin{equation} \lb{5.1}
\#(A_j)=3^{j^2} \quad \#(C_j) =2^{j^2} \quad \#(B_j)=\#(D_j)=j^6 -1
\end{equation}
On $A_j$, $a_n\equiv 1$, on $C_j$, $a_n\equiv\f12$, and on $B_j$ and
$D_j$, $\log(a_n^2)$ linearly interpolates from $\log(\f14)$ to $\log(1)$,
that is, for $n\in B_j$,
\begin{equation} \lb{5.2}
\f{a_n^2}{a_{n-1}^2} = c_j
\end{equation}
and for $n\in D_j$,
\begin{equation} \lb{5.3}
\f{a_{n-1}^2}{a_n^2} = c_j
\end{equation}
where
\begin{equation} \lb{5.3a}
c_j^{j^6} =\tfrac14
\end{equation}
so that
\begin{equation} \lb{5.3b}
1-c_j = kj^{-6} + o(j^{-6})
\end{equation}
for a suitable nonzero constant $k$. In particular,
\begin{equation} \lb{5.3c}
\abs{1-c_j} \leq E_0 j^{-6}
\end{equation}
for some $E_0$.
\end{example}

As in Example~\ref{E4.1}, this $J$ is regular with spectrum $[-2,2]$. We will prove that

\begin{theorem}\lb{T5.2} On $[-2,2]\setminus [-1,1]$, $J$ has purely singular spectrum and for Lebesgue
a.e.\ $x_0$ in this set, the Nevai condition fails. On $(-1,1)$, $J$
has purely a.c.\ spectrum and for all $x_0\in (-1,1)$, the Nevai
condition holds uniformly on compact subsets.
\end{theorem}

The technical key to the new elements of this example is

\begin{theorem}\lb{T5.3} For $x_0\in (-1,1)$, let $u_n (x_0,\theta_0)$ be the solution of \eqref{4.17}
{\rm{(}}with $a_0\equiv 1$ and $z=x_0${\rm{)}} for $n=1,2,\dots$, with
\begin{equation} \lb{5.4}
u_0=\cos\theta_0 \qquad u_1 =\sin\theta_0
\end{equation}
Then for any compact set $K\subset (-1,1)$, there is a constant, $C$, so that for all $x_0\in K$\!, all
$\theta_0$ and all $n$,
\begin{equation} \lb{5.5}
\abs{u_n(x_0,\theta_0)}\leq C
\end{equation}
\end{theorem}

\begin{proof}[Proof of Theorem~\ref{T5.2} given Theorem~\ref{T5.3}] $J$ has as one of its right limits,
$J_r$, the two-sided matrix with $b_n\equiv 0$, $a_n\equiv \f12$ whose a.c.\ spectrum is $\Sigma_\ac(J_r)
= [-1,1]$. By a theorem of Last--Simon \cite{S263}, $\Sigma_\ac(J)\subset [-1,1]$, so $J$ has purely
singular spectrum on $[-2,2]\setminus [-1,1]$. The results on this set for the Nevai condition follow
the arguments in Section~\ref{s4} without change.

Theorem~\ref{T5.3} implies that the transfer matrix $T_n(x_0)$ is
uniformly bounded in $n$ and $x_0\in K \subset (-1,1)$ compact.
Carmona's formula (see, e.g., \cite[Thm.~10.7.5]{OPUC2}; also
\cite{Carm,S263,Si253,Rev47}) then implies the spectrum is purely
a.c.\ on $(-1,1)$.

A bounded transfer matrix also implies $p_n(x_0)^2$ bounded above, and given constancy of the Wronskian
(i.e., $\det(T_n)=1$), uniform lower bounds on $p_n(x_0)^2 + p_{n+1}(x_0)^2$. Thus, on $(-1,1)$,
\begin{equation} \lb{5.6x}
\f{p_n(x_0)^2}{\sum_{j=0}^n p_j(x_0)^2} \leq \f{C}{n} \to 0
\end{equation}
proving \eqref{1.11}.
\end{proof}

The situation we need to control for Theorem~\ref{T5.3} has much in common with those studied by
Kiselev--Last--Simon \cite{KLS} and their techniques will work here. We note that in our situation,
$\sum_n (a_{n+1}-a_n)^2 + (b_{n+1}-b_n)^2<\infty$, a general condition studied recently by Denisov
\cite{Denppt}, but under the additional assumptions that $a_n \equiv 1$, $b_n\to 0$. It would be
interesting to see if one can extend his ideas to this context (see Conjecture~\ref{Con9.5} and the
discussion following it below).

We depend on the EFGP transform, as do \cite{KLS}, but we need to allow modification for our case where
$a_n$ is not identically one, as it is in \cite{KLS}. Since $b_n\equiv 0$ for us, we state the equations
for that case. One defines $R_n,\theta_n$ by
\begin{align}
R_n \sin (\theta_n) &= a_n u_n \sin(k_n) \lb{5.6}  \\
R_n \cos (\theta_n) &= a_n (u_{n+1}-u_n\cos(k_n)) \lb{5.7}
\end{align}
where $k_n$ is given by
\begin{equation} \lb{5.8}
2\cos(k_n) = \f{x_0}{a_n}
\end{equation}
We note, since $a_n \geq\f12$ and $\sup_{x_0\in K} \abs{x_0} <1$, that uniformly for $x_0\in K$ and all $n$,
\begin{equation} \lb{5.9}
\veps\leq k_n \leq \pi-\veps
\end{equation}
for some $\veps >0$ (depending on $K$).

As in \cite{KLS}, straightforward manipulations of the eigenfunction
equation show \eqref{4.17} is equivalent to
\begin{align}
\f{R_{n+1}^2}{R_n^2} &= 1+(a_{n+1}^2 - a_n^2)\, \f{\sin^2 (\theta_n + k_n)}{a_n^2 \sin^2 (k_n)} \lb{5.10} \\
\cot(\theta_{n+1}) &= \f{a_n}{a_{n+1}} \, \f{\sin(k_n)}{\sin (k_{n+1})} \, \cot(\theta_n + k_n) \lb{5.11}
\end{align}

$R_1$ and $\theta_1$ are functions of $\theta_0$ (given $a_1=a_2=1$) and $R_1$ is, for $x_0\in K$\!, uniformly
bounded above and below. Moreover, by \eqref{5.6} and \eqref{5.9}, for $C$ depending only on $K$\!,
\begin{equation} \lb{5.11a}
\abs{u_n} \leq CR_n
\end{equation}
Define
\begin{equation} \lb{5.12}
X_n = \f{(a_{n+1}^2 - a_n^2) \sin^2 (\theta_n + k_n)}{a_n^2 \sin^2 (k_n)}
\end{equation}

By Lemma~3.5 of \cite{KLS} and $\sup_n \abs{X_n} <\infty$, it suffices to prove that
\begin{equation} \lb{5.16}
\sup_N\, \biggl|\, \sum_{j=1}^N X_j\biggr| <\infty
\end{equation}

Define $\wti B_j, \wti D_j$ by adding to $B_j,D_j$ the index one
before (i.e., the top index of $A_j$ and $C_j$). Then $X_n$ is only
nonzero on $\cup_j (\wti B_j\cup \wti D_j)$. On $\wti B_j\cup\wti
D_j$, by \eqref{5.2}, \eqref{5.3}, and \eqref{5.3c},
\begin{equation} \lb{5.22}
\abs{X_n} \leq E(x_0) j^{-6}
\end{equation}
where
\begin{equation} \lb{5.23}
E(x_0) = E_0\, \sup_n \, \f{1}{\sin^2 (k_n)}
\end{equation}
is bounded above on $K$ by \eqref{5.9}. Theorem~\ref{T5.3} is reduced to proving
\begin{equation} \lb{5.29}
\sup_N \, \biggl|\, \sum_{n=1}^N X_n\biggr| <\infty
\end{equation}
uniformly in $\theta_1$ and $x_0\in K$.

Next, we note that one can write
\begin{equation} \lb{5.30}
X_n =  X_n^\sharp + \wti X_n
\end{equation}
using
\begin{equation} \lb{5.31x}
\sin^2 (\theta_n+k_n) =\tfrac12\, (1-\cos (2(\theta_n + k_n)))
\end{equation}

The $X_n^\sharp$ terms are independent of $\theta_n$ and there is a
symmetry between points in $\wti B_j$ and $\wti D_j$ which, given
the opposite signs of $1- a_n^2/a_{n+1}^2$, causes a partial
cancellation, that is, since $c_j^{-1} -c_j =O(j^{-6})$,
\begin{equation} \lb{5.31}
\sum_{n\in\wti B_j} X_n^\sharp + \sum_{n\in\wti D_j} X_n^\sharp =
O(j^{-6})
\end{equation}
Moreover,
\begin{equation} \lb{5.32}
\sum_{n\in\wti B_j}\, \abs{X_n^\sharp} \leq E(x_0)
\end{equation}

These together (plus the approximate cancellation) implies
\begin{equation} \lb{5.33}
\sup_N \, \biggl| \, \sum_{n=1}^N X_n^\sharp\biggr| \leq E(x_0) + O\biggl(\sum j^{-6}\biggr)
\end{equation}
so we have reduced the proof of \eqref{5.29}, and so of
Theorem~\ref{T5.3}, to proving
\begin{equation} \lb{5.34}
\sup_N\, \biggl| \, \sum_{n=1}^N \wti X_n\biggr| <\infty
\end{equation}
uniformly in $\theta_1$ and $x_0\in K$.

We want to use cancellations of sums of cosines---more explicitly, that sums of $M$ cosines with suitably
varying phase are of order $1$, not $M$\!. Here is what we need:

\begin{lemma}\lb{L5.6} For any $q\in (0,2\pi)$, any $\theta$, and $M$\!,
\begin{equation} \lb{5.35}
\biggl| \, \sum_{\ell=1}^M \cos(q\ell+\theta)\biggr| \leq \biggl[ \sin\biggl(\f{q}{2}\biggr)\biggr]^{-1}
\end{equation}
\end{lemma}

\begin{proof} Since $\cos(\psi)=\Real (e^{i\psi})$, it suffices to prove this if $\cos(q\ell+\theta)$ is
replaced by $e^{i(q\ell+\theta)}$. By summing a geometric series,
\begin{align*}
\biggl| \, \sum_{\ell=1}^M e^{i(q\ell+\theta)}\biggr|
&= \biggl| \f{e^{i[(M+1)q+\theta]} -e^{i[q+\theta]}}{e^{iq} -1} \biggr| \\
&\leq \f{2}{2\abs{(e^{iq/2} -e^{-iq/2})/2}} = \f{1}{\sin(\f{q}{2})}
\qedhere
\end{align*}
\end{proof}

In $\cos(\theta_n +k_n)$, both $k_n$ and $\theta_n$ are
$n$-dependent. But over subblocks small compared to $j^6$, $k_n$ is
close to constant and $\theta_{n+1}-\theta_n$ is close to constant.
Thus, we break $\wti B_j$ and $\wti D_j$ into $j^4$ blocks, each
with $j^2$ members, call them $\{\wti B_{j,\ell}\}_{\ell=1}^{j^4}$
and $\{\wti D_{j,\ell}\}_{\ell=1}^{j^4}$. For any $n$ in some $\wti
B_j$ or $\wti D_j$, let $\beta_n$ be the first element of the
subblock containing $n$ and
\[
\kappa_n = k_{\beta_n}
\]
Clearly, with constants uniformly bounded over $K$ (below, $C$ will stand for a generic constant bounded
on any compact $K\subset (-1,1)$),
\begin{align}
\abs{n-\beta_n} &\leq Cj^2 \lb{5.36} \\
\abs{k_n-\kappa_n} &\leq Cj^{-4} \lb{5.37}
\end{align}
\eqref{5.37} comes from the fact that over a subblock, $a_n$ changes by at most $j^2 O(j^{-6})$.

In \eqref{5.11}, the ratio of $a$'s is $1$ plus an error of order $j^{-6}$, so given that arc\,cot has
bounded derivatives,
\begin{equation} \lb{5.38}
\abs{\theta_{n+1} - (\theta_n + k_n)} \leq Cj^{-6}
\end{equation}
and so,
\begin{equation} \lb{5.39}
\abs{\theta_{n+1}-\theta_n -\kappa_n} \leq Cj^{-4}
\end{equation}

This implies that if
\begin{equation} \lb{5.40}
\ti\theta_n =\theta_{\beta_n} + (n-\beta_n)\kappa_n
\end{equation}
then
\begin{equation} \lb{5.41}
\abs{\theta_n - \ti\theta_n} \leq Cj^{-2}
\end{equation}
Define $Y_n$ to be $\wti X_n$ with  $\cos (2(\theta_n+k_n))$
replaced by $\cos(2(\ti\theta_n + \kappa_n))$ and $a_n^2 \sin^2
(k_n)$ by $a_{\beta_n}^2 \sin^2 (\kappa_n)$. By \eqref{5.37} and
\eqref{5.41} (since $a_{n+1}^2 -a_n^2 \sim j^{-6}$), on $\wti
B_j\cup\wti D_j$,
\begin{equation} \lb{5.42}
\abs{Y_n -\wti X_n} \leq Cj^{-8}
\end{equation}
and
\[
\sum_{n\in\wti B_j\cup\wti D_j}\, \abs{Y_n - \wti X_n} \leq Cj^{-2}
\]
which is summable in $j$. Thus, to prove \eqref{5.34}, we need
\begin{equation} \lb{5.43}
\sup_N\, \biggl| \sum_{n=1}^N Y_n\biggr| <\infty
\end{equation}
uniformly in $\theta_1$ and $K$\!.

\begin{proof}[Proof of Theorem~\ref{T5.3}] As noted, we are reduced to proving \eqref{5.43}. $\{1,\dots,N\}$
can be broken into sums over $0$ (i.e., $A_j$ and $C_j$, except for
their final indices), sums over some number of subblocks, and one
further partial subblock. Summing over a single subblock is, by
Lemma~\ref{L5.6} (given that, by \eqref{5.9}, $2\kappa_n$ is bounded
away from $0$ and $2\pi$), bounded by $Cj^{-6}$ (from the fact that
$a_{n+1}^2 - a_n^2\sim j^{-6})$. Since there are $2j^4$ subblocks in
$\wti B_j\cup\wti D_j$, we see that
\[
\biggl| \, \sum_{n=1}^N Y_n\biggr| \leq \sum_j (2j^4)(Cj^{-6}) <\infty
\qedhere
\]
\end{proof}

\section{The NTZ Argument} \lb{s6}

Here we begin with the key lemma of Nevai--Totik--Zhang \cite{NTZ91} and apply it to extend the result
of Zhang \cite{Zhang} to allow approach to an isospectral torus.

\begin{proposition}[\cite{NTZ91}]\lb{P6.1} For any positive $r$,
any $\theta,\varphi\in [0,2\pi]$, and $L$,
\begin{equation} \lb{6.1}
\f{12}{L}\, \sum_{j=0}^{L-1} \, \abs{1-re^{i(j\theta+\varphi)}}^2 \geq \abs{1-re^{i\varphi}}^2
\end{equation}
\end{proposition}

\begin{remarks} 1. \cite{NTZ91} allow general $p>0$ where we take $p=2$; but for $p=2$, their constant
is $32$, not $12$.

\smallskip
2. We include a proof at the end of this section for the reader's convenience and because we want to
emphasize the concepts in the context of what we cannot do in the next section.

\smallskip
3. For $\theta$ not near $0$ or $2\pi$, the idea behind a bound of this form is the same as the idea
behind Lemma~\ref{L5.6}. As $\theta\to 0$ or $2\pi$, for this argument to work, the constant $12$ has
to be replaced by larger and larger numbers. The idea for small $\theta$ is instead to use the fact that
enough terms need to be close to the initial one.
\end{remarks}

\begin{corollary}[\cite{Zhang}] \lb{C6.2} Let $A$ be a $2\times 2$ matrix with
\begin{equation} \lb{6.2}
\det(A)=1 \qquad \abs{\tr(A)} \leq 2
\end{equation}
Then for any vector $v\in\bbC^2$ {\rm{(}}with $v=(v_1,v_2)$ the components of $v${\rm{)}},
\begin{equation} \lb{6.3}
\abs{(A^{L-1} v)_1}^2 \leq \f{12}{L}\, \sum_{j=0}^{L-1}\, \abs{(A^j v)_1}^2
\end{equation}
\end{corollary}

\begin{proof} If $A$ obeys \eqref{6.2}, so does $B = A^{-1}$, and if $w=A^{L-1} v$, \eqref{6.3} is
equivalent to
\begin{equation} \lb{6.4}
\abs{w_1}^2 \leq \f{12}{L}\, \sum_{j=0}^{L-1}\, \abs{(B^j w)_1}^2
\end{equation}
so we need only prove \eqref{6.4}.

Any $B$ obeying \eqref{6.2} is a limit of $B$'s with $\abs{\tr(B)} <2$, so we can suppose
\begin{equation} \lb{6.5}
\det(B)=1 \qquad \abs{\tr(B)} <2
\end{equation}
In that case, $B$ is diagonalizable and has eigenvalue $e^{\pm i\theta}$ with $2\cos(\theta) =\tr(B)$.

In particular, for any $v$,
\begin{equation} \lb{6.6}
(B^\ell v)_1 = \alpha e^{i\ell\theta} + \beta e^{-i\ell\theta}
\end{equation}
for some $\alpha,\beta$. By replacing $\theta$ by $-\theta$, we can suppose $\alpha\neq 0$ (if
$\alpha =\beta=0$, \eqref{6.4} is trivial!). Write $-\beta/\alpha = re^{-i\varphi}$. Then \eqref{6.4}
is equivalent to (after multiplying by $\abs{\alpha}^{-2}$)
\begin{equation} \lb{6.7}
\abs{1-re^{-i\varphi}}^2 \leq \f{12}{L}\, \sum_{j=0}^{L-1}\, \abs{1-re^{-i(2j\theta+\varphi)}}^2
\end{equation}
which, after a change of names of $\theta,\varphi$, is \eqref{6.1}.
\end{proof}

Recall that if $\{a_n,b_n\}_{n=1}^\infty$ are Jacobi parameters, a two-sided set $\{a_n^{(r)},
b_n^{(r)}\}_{n=-\infty}^\infty$ is called a right limit if for some $m_j\to\infty$ and all $n=0,
\pm 1,\dots$,
\begin{equation} \lb{6.8x}
a_{m_j+n} \to a_n^{(r)} \qquad b_{m_j+n}\to b_n^{(r)}
\end{equation}
If $\sup_n (\abs{a_n} + \abs{b_n}) < \infty$, there are right limits by compactness and, indeed, any sequence
$m_k$ has a subsequence defining a right limit. Right limits are described in \cite[Ch.~7]{Rice} and
references quoted there.

Any finite gap set
\begin{equation} \lb{6.8}
\fre = [\alpha_1,\beta_1] \cup \cdots \cup [\alpha_{\ell+1}, \beta_{\ell+1}]
\end{equation}
with
\begin{equation} \lb{6.9}
\alpha_1 < \beta_1 < \alpha_2 < \beta_2 < \cdots < \beta_{\ell+1}
\end{equation}
defines an $\ell$-dimensional isospectral torus, $\calT_\fre$, of almost periodic two-sided Jacobi
matrices, $J$ with $\sigma(J)=\fre$. $\calT_\fre$ can be defined using minimal Herglotz functions
(\cite[Ch.~5]{Rice}) or reflectionless requirements (\cite[Ch.~7]{Rice}). If $\rho_\fre$ is the
potential theoretic equilibrium measure for $\fre$ (see, e.g., \cite{StT, EqMC}), we say $\fre$ is
``periodic'' if and only if each $\rho_\fre ([\alpha_j,\beta_j])$ is rational; equivalently,
all $J\in\calT_\fre$ have a common period $p$.

The Nevai class for $\fre$ is defined to be those one-sided $J$'s whose right limits are all in
$\calT_\fre$. For $\fre=[-2,2]$, $\calT_\fre$ has a single point (with period $1$!) and the Nevai
class for $\fre$ is the usual Nevai class.

\begin{theorem}\lb{T6.3} If $J$ lies in the Nevai class for a periodic $\fre$, then the Nevai condition
holds uniformly for $J$ on $\fre$.
\end{theorem}

\begin{remark} If $J$ has a single $p$ element orbit, $J^{(r)}\in\calT_\fre$, as right limits (i.e., $J$
is asymptotically periodic), this is a result of \cite{Zhang,Szwarc2}.
\end{remark}

\begin{proof} Let $p$ be the period of $\fre$. We will prove that any $x_n\in\fre$ and any $L$
\begin{equation} \lb{6.10}
\limsup_{n\to\infty} \, \f{\abs{p_n(x_n)}^2}{\sum_{j=n-pL}^n \abs{p_j(x_n)}^2} \leq \f{12}{L}
\end{equation}
from which
\begin{equation} \lb{6.11}
\limsup_{n\to\infty}\, \f{\abs{p_n(x_n)}^2}{\sum_{j=0}^n \abs{p_j(x_n)}^2} =0
\end{equation}
proving the claimed uniform Nevai condition.

Without loss, we can pass to a subsequence so that $x_n\to
x_\infty$, so that the ratio in \eqref{6.10} still converges to the
$\limsup$, so that $a_{n+k}\to a_k^{(r)}$, $b_{n+k}\to b_k^{(r)}$
for some periodic right limit and so that $(p_n(x_n),
p_n(x_{n-1}))/\norm{(p_n(x_n), p_{n-1}(x_{n-1}))}$ has a limit in
$\bbC^2$.

The transfer matrix over $p$ units starting at $0$ for that $x_\infty$ is a matrix $A$ obeying \eqref{6.2}.
So, by \eqref{6.3},
\begin{equation} \lb{6.12}
\limsup_{n\to\infty}\, \f{\abs{p_n(x_n)}^2}{\sum_{j=0}^{L-1} \abs{p_{n-jp}(x_n)}^2} \leq \f{12}{L}
\end{equation}
which implies \eqref{6.10}.
\end{proof}

We turn to the proof of Proposition~\ref{P6.1}. Without loss, we can (by taking complex conjugates) suppose
\begin{equation} \lb{6.13}
0<\theta \leq \pi
\end{equation}
(since $\theta=0$ is trivial). There are three cases to consider:

\smallskip
\noindent\ul{Case 1}. \ $L\leq 12$, which is trivial.

\smallskip
\noindent\ul{Case 2}.
\begin{equation} \lb{6.14}
\theta L \geq 2\pi \qquad L\geq 13
\end{equation}

\smallskip
\noindent\ul{Case 3}.
\begin{equation} \lb{6.15}
\theta L < 2\pi
\end{equation}

\begin{proof}[Proof of Proposition~\ref{P6.1}] Consider Case~2 first,
expanding
\begin{equation} \lb{6.16}
\sum_{j=0}^{L-1}\, \abs{1- re^{i(j\theta +\varphi)}}^2 = L(1+r^2) -2r \Real [X]
\end{equation}
where
\begin{equation} \lb{6.17}
X=\f{e^{i(L\theta+\varphi)}- e^{i\varphi}}{e^{i\theta}-1}
\end{equation}
so
\begin{equation} \lb{6.18}
\abs{X} \leq \f{1}{\f12 \abs{1-e^{i\theta}}} = \f{1}{\abs{\sin(\f{\theta}{2})}}
\leq \f{\pi}{\theta}
\end{equation}
since
\begin{equation} \lb{6.19}
\inf_{0\leq y\leq \pi}\, \biggl[\f{\sin(\f{y}{2})}{y}\biggr] = \f{1}{\pi}
\end{equation}

By \eqref{6.14}, $\pi/\theta \leq L/2$, so by \eqref{6.16},
\begin{align}
\text{LHS of \eqref{6.1}} & \geq L(1+r^2)-Lr \notag \\
&\geq \f{L}{2}\, (1+r^2) \notag \\
& > 6 (1+r^2) \lb{6.20}
\end{align}
(since $L>12$). Clearly,
\begin{equation} \lb{6.21}
\text{RHS of \eqref{6.1}} \leq \abs{1+r}^2 \leq 2(1+r^2)
\end{equation}
so \eqref{6.1} holds in Case~2.

That leaves Case~3. We will consider $\varphi <0$ ($\varphi >0$ is even easier). Consider the $L$
points
\begin{equation} \lb{6.22}
T=\{\varphi+j\theta\}_{j=0}^{L-1}
\end{equation}
Since $L\theta <2\pi$, they do not make it back around the circle.
Consider the three sets: $S_1 = \{\eta\mid\varphi\leq \eta <
\f{\varphi}{2}\}$, $S_2 = \{\f{\varphi}{2} \leq \eta < 0\}$, and
$S_3 = \{0\leq\eta < -\f{\varphi}{2}\}$. Clearly, $\#(S_1\cap T)
\geq \max(\#(S_2\cap T), \#(S_3\cap T))$, so at most two-thirds of
the points in $T$ lie in $S_2\cup S_3$.

By the lemma below, if $\eta\in T\setminus (S_2\cup S_3)$,
\begin{equation} \lb{6.23}
\abs{1-re^{i\eta}}^2 \geq \tfrac14\, \abs{1-re^{i\varphi}}^2
\end{equation}
so
\[
\text{LHS of \eqref{6.1}} \geq \f{L}{3}\biggl(\f{12}{L}\biggr)\, \f14\, \abs{1-re^{i\varphi}}^2
= \text{RHS of \eqref{6.1}}
\qedhere
\]
\end{proof}

\begin{lemma} \lb{L6.4}
\begin{equation} \lb{6.24}
\inf_{\substack{ \pi \geq \abs{\eta}\geq \abs{\f{\varphi}{2}} \\ 0 < r}} \,
\f{\abs{1-re^{i\eta}}}{\abs{1-re^{i\varphi}}} \geq \f12
\end{equation}
\end{lemma}

\begin{proof} $\abs{1-re^{i\eta}}/\abs{1-re^{i\varphi}}$ is invariant under $r\to r^{-1}$, so we can
suppose $0<r \leq 1$. Moreover, $\abs{1-re^{i\eta}}$ is invariant
under $\eta\to -\eta$ and increasing in $\eta$ for $0<\eta <\pi$, so
the $\inf$ occurs at $\eta= \f{\varphi}{2}$.

A straightforward calculation shows
$\abs{1-re^{i\varphi/2}}/\abs{1-re^{i\varphi}}$ is decreasing in $r$
in $r\in (0,1]$, so the $\inf$ is
$\abs{1-e^{i\varphi/2}}/\abs{1-e^{i\varphi}} = \abs{\sin
(\f{\varphi}{4})}/\abs{\sin(\f{\varphi}{2})} = 1/\abs{2\cos
(\f{\varphi}{4})} \geq \f12$.
\end{proof}

\section{The Nevai Class of a General Finite Gap Set} \lb{s7}

In this section, we will discuss the extension of Theorem~\ref{T6.3} to general finite gap sets. We
will only be able to prove the weaker result that the Nevai condition holds uniformly on compact
subsets of $\fre^\intt$. In the next section, using different methods, we will prove the result
uniformly on all of $\fre$.

We begin by noting the following abstraction of the argument we used in the proof of Theorem~\ref{T6.3}:

\begin{proposition}\lb{P7.1} Let $J$ be a half-line Jacobi matrix and let $\calR$ be the set of its right
limits. Let $K\subset\bbR$ be a compact set. For $v\in\bbC^2$ and
$J^{(r)}\in\calR$, let $u_n (v,J^{(r)},z)$ solve
\begin{equation} \lb{7.1}
a_n^{(r)} u_{n+1} + b_n^{(r)} u_n + a_{n-1}^{(r)} u_{n-1} = zu_n
\end{equation}
with
\begin{equation} \lb{7.2}
(u_0,u_1) = (v_1, v_2)
\end{equation}
Suppose that for all $\veps$, there is $N$\! so that for all unit
vectors $v\in\bbC^2$, all $J^{(r)}\in\calR$, all $x_0\in K$\!, and
all $n>N$\!,
\begin{equation} \lb{7.3}
\f{\abs{u_n(v,J^{(r)},x_0)}^2}{\sum_{j=0}^n \abs{u_j
(v,J^{(r)},x_0)}^2} \leq \veps
\end{equation}
Then $J$ obeys the Nevai condition uniformly on $K$\!.
\end{proposition}

\begin{proposition}\lb{P7.2} Let $J$ be a half-line Jacobi matrix obeying \eqref{1.13a}
and let $\calR$ be the set of its right limits. Suppose that there
is a compact subset $K\subset\bbR$ such that for each $x_0\in K$ and
$J^{(r)}\in\calR$, there is a solution $u_n^+(J^{(r)},x_0)$ of
\eqref{7.1} {\rm{(}}with $z=x_0${\rm{)}} so that
\begin{alignat}{2}
&\text{\rm{(i)}} \qquad && \sup_{n,x_0,J^{(r)}}\, \abs{u_n^+ (J^{(r)}, x_0)} <\infty \lb{7.4} \\
&\text{\rm{(ii)}} \qquad && \inf_{x_0,J^{(r)}}\, a_0^{(r)} \abs{
u_1^+ \ol{u_0^+} - \ol{u_1^+}\,  u_0^+} >0 \lb{7.5}
\end{alignat}
Then the Nevai condition holds for $J$ uniformly on $K$\!.
\end{proposition}

\begin{remark} These are very strong conditions, but they hold in the finite gap case.
\end{remark}

\begin{proof} Define
\begin{equation} \lb{7.6}
U_n (J^{(r)},x_0) = \f{1}{d(J^{(r)},x_0)}
\begin{pmatrix}
u_{n+1}^+  & \ol{u_{n+1}^+} \\
a_n^{(r)} u_n^+ & a_n^{(r)}\, \ol{u_n^+}
\end{pmatrix}
\end{equation}
where $d(J^{(r)},x_0)$ is a square root of
\begin{equation} \lb{7.7}
a_n^{(r)} (u_{n+1}^+ \,\ol{u_n^+} - \ol{u_{n+1}^+} \, u_n^+)
\end{equation}
which is $n$-independent. Then $U_n$ is uniformly bounded in $x_0\in
K$\!, $J^{(r)}\in\calR,n$ by \eqref{7.4}/\eqref{7.5} and has
determinant $1$, so the same is true of $U_n^{-1}$.

Moreover, the transfer matrix for $J^{(r)}$ is
\begin{equation} \lb{7.8}
T_n = U_n U_0^{-1}
\end{equation}
so it is bounded in $n,J^{(r)},x$, and has a bounded inverse. This
shows
\begin{equation} \lb{7.9}
\abs{u_n (v, J^{(r)}, x_0)}^2 + \abs{u_{n+1} (v,J^{(r)}, x_0)}^2
\end{equation}
is uniformly bounded above and below as $v$ runs through unit vectors.

The ratio in \eqref{7.3} is thus uniformly bounded by $c/n$, so Proposition~\ref{P7.1} is applicable.
\end{proof}

\begin{theorem}\lb{T7.3} If $J$ lies in the Nevai class for a finite gap set $\fre$, then the Nevai
condition holds uniformly on compact subsets of $\fre^\intt$\!.
\end{theorem}

\begin{proof} In \cite{CSZ1} (see also \cite[Ch.~9]{Rice}), Jost solutions are constructed on the
isospectral torus, $\calT_\fre$, that obey \eqref{7.4}/\eqref{7.5}.
\end{proof}

\section{Absence of Pure Points in Right Limits} \lb{snew8}

In this section, we want to note and apply the following:

\begin{theorem}\lb{Tn8.1} Let $J$ be a bounded half-line Jacobi matrix with \eqref{1.13a} and
let $\calR$ be the set of its right
limits. Let $\Xi$ be the set of $x_0\in\bbR$ so that for every
$J^{(r)}\in\calR$ and every nonzero solution $u_n$ of \eqref{7.1}
with $z=x_0$, we have
\begin{equation} \lb{n8.1}
\sum_{n=-\infty}^0\, \abs{u_n}^2 =\infty
\end{equation}
Then
\begin{SL}
\item[{\rm{(i)}}] The Nevai condition holds uniformly on any compact subset of $\Xi$.
\item[{\rm{(ii)}}] If $\Xi$ contains $\sigma_\ess (J)$, then the Nevai condition holds uniformly on
$\sigma(J)$.
\end{SL}
\end{theorem}

We will provide a proof below. We first discuss some consequences.

\begin{theorem}\lb{Tn8.2} If $J$ lies in the Nevai class for a finite gap set $\fre$, then the Nevai condition
holds uniformly on $\sigma(J)$.
\end{theorem}

\begin{proof} In \cite{CSZ1} (see also \cite[Ch.~9]{Rice}), it is proven that for any $J^{(r)}$ in the
isospectral torus, $\calT_\fre$, and any $x_0\in \fre^\intt$, every solution is almost periodic; and for $x_0
\in\{\alpha_j,\beta_j\}_{j=1}^{\ell+1}$, every solution is the sum of an almost periodic function and $n$
times an almost periodic function. Nonzero almost periodic functions obey \eqref{n8.1} and $\sigma_\ess (J)
=\fre$, so Theorem~\ref{Tn8.1} is applicable.
\end{proof}

There is a class of whole-line stochastic Jacobi matrices called
subshifts, with work reviewed in \cite{Dam}. The most famous is the
Fibonacci model which has $(\chi_I=$ characteristic function of the
set $I$)
\[
a_n\equiv 1 \qquad b_n\equiv \chi_{[1-\alpha, 1]} ((n\alpha + \theta)_{\text{mod}\, 1})
\]
where $\alpha =\f12 (\sqrt{5}-1)$ and $\theta$ is a parameter (e.g.,
$0$). The name comes from the fact that the transfer matrix, $T_n$,
has special properties when $n$ is a Fibonacci number. Damanik--Lenz
\cite{DL} showed that there are no solutions $\ell^2$ at $-\infty$
for any $\theta$ and any $x_0$ in the spectrum, and it is not hard
to see the right limits for any half-line Fibonacci problem are
again Fibonacci models or such models modified at a single site.
Thus, Theorem~\ref{Tn8.1} is applicable, and

\begin{theorem}\lb{Tn8.2a} Any Fibonacci model restricted to a half line obeys the Nevai condition uniformly
on the spectrum.
\end{theorem}

\begin{remark} Results of Damanik--Killip--Lenz \cite{DKL} allow extension of this to general Sturmian models.
\end{remark}

Theorem~\ref{Tn8.1} was motivated by our trying to understand Szwarc \cite{Szwarc2}. He noted that one could
use results of Nevai \cite{Nev79} on weak limits of the measure
\begin{equation} \lb{n8.3}
\sum_j \, \f{p_n(x_j^{(n+1)})^2}{\sum_{k=0}^n p_k (x_j^{(n+1)})^2} \, \delta_{x_j^{(n+1)}}
\end{equation}
where $x_j^{(n+1)}$ are the solutions of
\begin{equation} \lb{n8.4}
p_{n+1}(x_j^{(n+1)}) =0
\end{equation}
that is, the ratios in \eqref{1.11} are weights in some natural measures. Thus, a failure of \eqref{1.11} should
imply a suitable half-line limit has a pure point and that is what is forbidden by \eqref{n8.1}. We begin with:

\begin{lemma} \lb{Ln8.3} Let $J_{n;F}$ be the $n\times n$ truncated Jacobi matrix. Then the spectral measure of
$J_{n+1;F}$ and vector $\delta_{n+1}$ is \eqref{n8.3} where the $x_j^{(n+1)}$ solve \eqref{n8.4}.
\end{lemma}

\begin{proof} It is well known  (see \cite{Rice}) that
\begin{equation} \lb{n8.5}
\det(x-J_{n+1;F}) =P_{n+1}(x)
\end{equation}
so the eigenvalues are the solutions of \eqref{n8.4}. The unnormalized eigenvector for $x_j^{(n+1)}$ is
$v_k$, given by
\begin{equation} \lb{n8.6}
v_k = p_{k-1} (x_j^{(n+1)})
\end{equation}
so \eqref{n8.3} has the squares of the normalized components for $\delta_{n+1}$.
\end{proof}

\begin{lemma} \lb{Ln8.4} Let $J_F^{(n)}$ be a family of $m_n\times m_n$ finite Jacobi matrices with coefficients
associated to $\{a_j^{(n)}\}_{j=1}^{m_n-1} \cup\{b_j^{(n)}\}_{j=1}^{m_n}$. Suppose
\begin{SL}
\item[{\rm{(i)}}] $m_n\to\infty$
\item[{\rm{(ii)}}]
\begin{equation} \lb{n8.7}
\sup_{j,n}\, \abs{a_j^{(n)}} + \abs{b_j^{(n)}} <\infty
\end{equation}
\item[{\rm{(iii)}}] For each fixed $j$,
\begin{equation} \lb{n8.8}
a_j^{(n)} \to a_j^{(\infty)} \qquad b_j^{(n)} \to b_j^{(\infty)}
\end{equation}
\end{SL}
Let $J^{(\infty)}$ be the infinite Jacobi matrix with parameters $\{a_j^{(\infty)},b_j^{(\infty)}\}_{j=1}^\infty$.
Let $d\rho^{(n)}$ be the spectral measure for $J_F^{(n)}$ and $\delta_1$, and $d\rho^{(\infty)}$ for $J^{(\infty)}$
and $\delta_1$. Then
\begin{equation} \lb{n8.9}
\wlim_{n\to\infty}\, d\rho^{(n)} =d\rho^{(\infty)}
\end{equation}
\end{lemma}

\begin{remark} This generalizes Theorem~3 in \cite[p.~17]{Nev79}.
\end{remark}

\begin{proof} Extend $J_F^{(n)}$ to an infinite matrix by setting all other matrix elements to $0$. Then
\eqref{n8.7} implies that
\begin{equation} \lb{n8.10}
\sup_n\, \norm{J_F^{(n)}} <\infty
\end{equation}
and \eqref{n8.8} implies that for any finite support vector, $v$,
\begin{equation} \lb{n8.11}
\norm{(J_F^{(n)} -J^{(\infty)})v} \to 0
\end{equation}
It follows that
\begin{equation} \lb{n8.12}
\slim\, J_F^{(n)} = J^{(\infty)}
\end{equation}
So, by \eqref{n8.10},
\begin{equation} \lb{n8.13}
\slim\, (J_F^{(n)})^k = (J^{(\infty)})^k
\end{equation}
for all $k$. Thus,
\begin{equation} \lb{n8.14}
\lim \jap{\delta_1, (J_F^{(n)})^k \delta_1} = \jap{\delta_1, (J^{(\infty)})^k \delta_1}
\end{equation}
so
\begin{equation} \lb{n8.15}
\lim \int x^k\, d\rho^{(n)} = \int x^k \, d\rho^{(\infty)}
\end{equation}
from which \eqref{n8.9} follows.
\end{proof}

\begin{proof}[Proof of Theorem~\ref{Tn8.1}] Let $\wti K$ be a compact subset of $\Xi$.
If \eqref{1.11} does not
hold uniformly, we can find $n(j)\to\infty$, $x_j\in\wti K$, and $\veps >0$ so that for all $j$,
\begin{equation} \lb{n8.16}
\f{p_{n(j)} (x_j)^2}{\sum_{k=0}^{n(j)} p_k (x_j)^2} \geq \veps
\end{equation}
By passing to a subsequence, we can suppose
\begin{equation} \lb{n8.17}
x_j \to x_\infty \in \wti K
\end{equation}

Notice that \eqref{n8.16} implies
\begin{equation} \lb{n8.18}
p_{n(j)-1} (x_j)^2 \leq \sum_{k=0}^{n(j)} p_k (x_j)^2 \leq \veps^{-1} p_{n(j)} (x_j)^2
\end{equation}
Define $\ti b_{n(j)+1}$ by
\begin{equation} \lb{n8.19}
a_{n(j)} p_{n(j)-1} (x_j) + (\ti b_{n(j)+1}-x_j) p_{n(j)} (x_j) =0
\end{equation}
Thus, for the OPs with Jacobi parameters $(a_1, \dots, a_{n(j)})$, $(b_1, \dots, b_{n(j)}, \ti b_{n(j)+1})$,
we have
\begin{equation} \lb{n8.20}
\ti p_{n(j)+1} (x_j)=0
\end{equation}
Moreover, by \eqref{n8.19} and \eqref{n8.18},
\begin{equation} \lb{n8.21}
\limsup\, \abs{\ti b_{n(j)+1}} \leq \veps^{-1/2} \sup_k\, \abs{a_k} + \sup_k\, \abs{x_k}
\end{equation}
is finite.

Let $J_F^{(j)}$ be the Jacobi matrix which is $(n(j)+1)\times (n(j)+1)$ with $a_{n(j)}, a_{n(j)-1}, \dots, a_1$
off diagonal and $\ti b_{n(j)+1} - x_j +x_\infty, b_{n(j)} -x_j +x_\infty, \dots, b_1 -x_j +x_\infty$ on diagonal.
By \eqref{n8.16} and Lemma~\ref{Ln8.3} (turning $J_{n(j)+1;F}$ on its head!), the spectral measure for $J_F^{(j)},
\delta_1$ has a pure point at $x_\infty$ of mass at least $\veps$.

By passing to a further subsequence, we can suppose for all $q$ that
$a_{n(j)+q}\to a_q^{(r)}$, $b_{n(j)+q} \to b_q^{(r)}$ for some right
limit, $J^{(r)}$, and, by \eqref{n8.21} and a further subsequence,
that $\ti b_{n(j)+1}\to \ti b_1^{(r)}$.

The coefficients of $J_F^{(j)}$ clearly obey \eqref{n8.7} and there
is a $J^{(\infty)}$ so \eqref{n8.8} holds. This is given by the
reversed left side of $J^{(r)}$ (from $-\infty$ to $1$), with
$b_1^{(r)}$ replaced by $\ti b_1^{(r)}$.

For any positive function $f$,
\begin{equation} \lb{n8.22}
\int f(y)\, d\rho^{(n)}(y) \geq \veps f(x_\infty)
\end{equation}
So, by \eqref{n8.9},
\begin{equation} \lb{n8.23}
\int f(y)\, d\rho^{(\infty)} (y) \geq \veps f(x_\infty)
\end{equation}
Thus,
\begin{equation} \lb{n8.24}
\rho^{(\infty)} (\{x_\infty\}) \geq \veps
\end{equation}
and $x_\infty$ is an eigenvalue of $J^{(\infty)}$. Thus, $J^{(r)}$ has an eigensolution at $x_\infty$ which is
$\ell^2$ at $-\infty$, violating \eqref{n8.1}. This proves statement (i) of Theorem~\ref{Tn8.2}.

For (ii), by passing to a subsequence, \eqref{n8.17} holds for some $x_\infty$. If $x_\infty$ is an isolated
eigenvalue, $x_j$ must be equal to $x_\infty$ for large $j$, so $\limsup\abs{p_j(x_\infty)}>0$, violating
the condition that $x_\infty$ is an isolated eigenvalue. If $x_\infty \in\sigma_\ess(J)$, the argument in
the first part produces a contradiction. Thus, (ii) is proven.
\end{proof}

\section{Some Comments} \lb{s8}

We end this paper with some final results and comments. The following must be well known in the ergodic
theory community:

\begin{proposition} \lb{P8.1} Let $a_n$ be a sequence of reals and
\begin{equation} \lb{8.1}
C_n = \f{1}{n}\, \sum_{j=1}^n a_j
\end{equation}
If
\begin{equation} \lb{8.2x}
\lim_{n\to\infty}\, C_n = C_\infty \neq 0
\end{equation}
exists, then
\begin{equation} \lb{8.3x}
\f{a_n}{n C_n} \to 0
\end{equation}
\end{proposition}

\begin{proof} Since
\begin{align}
a_n &= n C_n -(n-1)C_{n-1} \lb{8.2}\\
n^{-1} a_n &= C_n - C_{n-1} + n^{-1} C_{n-1} \lb{8.3}
\end{align}
so
\begin{equation} \lb{8.4}
\f{a_n}{n C_n} = \f{C_n - C_{n-1} + n^{-1} C_{n-1}}{C_n}
\end{equation}
goes to zero if \eqref{8.2x} holds.
\end{proof}

It is an idea associated especially with Freud and Nevai (see \cite{NevFr}) that on the a.c.\ spectrum,
\begin{equation} \lb{8.5}
\f{1}{n}\, \sum_{j=0}^{n-1} p_j(x_0)^2 \to \f{\rho_\infty (x_0)}{w(x_0)}
\end{equation}
where $\rho_\infty$ is the density of zeros and $w$ is the a.c.\ part of the underlying measure, with
concrete results both classical (\cite{MNT91,Tot}) and recent (\cite{ALS,Lub,Lub-jdam,2exts,Tot-prep}).

Via Proposition~\ref{P8.1}, this gives several results on the Nevai condition. Totik's result \cite{Tot}
implies:

\begin{theorem}\lb{T8.2} Let $d\rho$ have the form
\begin{equation} \lb{8.6}
d\rho(x) = w(x)\, dx + d\rho_\s (x)
\end{equation}
with $d\rho_\s$ Lebesgue singular. Let $\fre$ be the essential support of $d\rho$ and suppose $\rho$
is regular for $\fre$. Suppose $I$ is an open interval in $\fre$ on which a Szeg\H{o} condition holds:
\begin{equation} \lb{8.7}
\int_I \log(w(x))\, dx >-\infty
\end{equation}
Then for Lebesgue a.e.\ $x_0\in I$, the Nevai condition holds.
\end{theorem}

More recent work on uniform convergence \cite{Lub,2exts} implies

\begin{theorem}\lb{T8.3} Under the hypotheses of Theorem~\ref{T8.2}, if \eqref{8.7} is replaced by
$w$ continuous on $I$ and $\inf_I w(x) >0$, then the Nevai condition holds on all of $I$ uniformly
on compact subsets of $I$.
\end{theorem}

Recent work on ergodic Jacobi matrices (\cite{ALS}; see that paper for the definition of ergodic Jacobi
matrices) implies

\begin{theorem} \lb{T8.4} Let $J_\omega$ be a family of ergodic Jacobi matrices obeying \eqref{1.13a}.
Suppose $\Sigma_\ac$, the essential support of the a.c.\ spectrum, is nonempty.
Then for a.e.\ $\omega, x_0\in\Sigma_\ac$, the
Nevai condition holds.
\end{theorem}

We also want to make a comment regarding the unbounded case (where
$\rho$ is not compactly supported). In this case, there exist
measures, $\rho$, of various types (including pure point!) with $
\lim_{n \rightarrow \infty}\f{a_{n+1}^2 [p_n(x_0)^2 +
p_{n+1}(x_0)^2]}{K_n(x_0,x_0)} \neq 0$ for $x$ in a set of positive
$\rho$ measure. Indeed, the power law behavior of the generalized
eigenfunctions in \cite{breuer} and those associated with the
absolutely continuous part of the measure in \cite{DN} imply the
limit in these cases is actually $\infty$ for certain values of the
relevant parameters. In the Introduction, it was noted that this
means that truncations of the generalized eigenfunction do not form a
sequence of approximate eigenvectors. Although \eqref{2.14}
is still true in this case, Theorem~\ref{T1.2} does not hold for
measures that are not compactly supported, so this does not
constitute a counterexample to Conjecture~\ref{Con1.4}. In any case,
it seems to be an interesting challenge to study the Nevai condition
in the unbounded case.

Next, we turn to some remarks on Conjecture~\ref{Con1.4}. It is
standard to break $d\rho$ into three parts: a.c., pure point, and
singular continuous. For compactly supported measures the Nevai
condition always holds at pure points, since if $x_0$ is a pure
point, $\sum_{j=0}^\infty p_j(x_0)^2 <\infty$, so $p_n(x_0)^2 \to
0$. For $d\rho_{\rm c}$($=$ the continuous part of $d\rho$) it is not
hard to see that $\lim_{n \rightarrow \infty}\int
\f{p_n(x)^2}{\sum_{j=0}^n p_j(x)^2}d\rho_{\rm c}(x)=0$ so the issue
is going from convergence of integrals to pointwise convergence.

If $x_0$ is not a pure point but $\sup_n \abs{p_n(x_0)}<\infty$, then since $\sum_{j=0}^\infty p_j(x_0)^2
=\infty$, the Nevai condition holds. Thus, our Conjecture~\ref{Con1.4} is related to a famous conjecture
of Steklov and what is sometimes called the Schr\"odinger conjecture \cite{MMG}. There are known
counterexamples to both conjectures (see \cite{Rakh} and \cite{Jito}), but the counterexample in
\cite{Rakh} has failure of boundedness at a single point and the counterexample in \cite{Jito} does not
seem to violate the Nevai condition, so Conjecture~\ref{Con1.4} could be true.

We note that the currently open version of the Schr\"odinger
conjecture, that for a.e.\ $x_0$ in the essential support of the a.c.\ spectrum one has
bounded eigenfunctions, would imply the Nevai condition a.e.\ on the
essential support of the a.c.\ spectrum.

Finally, in connection with the example of Section~\ref{s5}, we would like to point out
that we believe the following is true:
\begin{conjecture}\lb{Con9.5}
Let $q\in\bbN$, let $\{a_n, b_n\}_{n=1}^{\infty}$ be Jacobi parameters
obeying
\begin{equation} \lb{9.10}
\sum_{n=1}^{\infty} (a_{n+q}-a_n)^2 + (b_{n+q}-b_n)^2 < \infty
\end{equation}
and let $\calR$ be the set of corresponding right limits.
Then
\[
\Sigma_{\ac} = \bigcap_{\calR} \sigma(J^{(r)})
\]
where $\Sigma_{\ac}$
is the corresponding
essential support of the a.c.\ spectrum.
\end{conjecture}

We note that \eqref{9.10} implies $\calR$ is made of $q$-periodic
Jacobi matrices. The inclusion
$\Sigma_{\ac} \subset \bigcap_{\calR} \sigma(J^{(r)})$
follows from a general result of \cite{S263},
so the point here is the inclusion in the other
direction. Conjecture~\ref{Con9.5} generalizes a conjecture of
Kaluzhny--Last \cite{KaLa}, who make this conjecture for
the special case where $\calR$ consists of a single element.
Denisov's result \cite{Denppt} establishes it for the special
where $a_n\equiv 1$ and where the single element of $\calR$
is the free Jacobi matrix (namely, $b_n \to 0$), proving
an even earlier variant of this conjecture by Last \cite{Las}. 
Theorem~\ref{T5.2} shows that Conjecture~\ref{Con9.5}
holds for the Jacobi matrix of Example~\ref{E5.1}
(a special case where \eqref{9.10} holds for $q=1$), thus providing
some level of confirmation for it.

\bigskip

\end{document}